\documentclass[oneside, a4paper, final]{amsart}

\usepackage[british]{babel}
\usepackage[english=british]{csquotes}
\usepackage{tikz-cd}
\usepackage{amsmath}
\usepackage{amssymb}
\usepackage{amsthm}
\usepackage{mathtools}
\usepackage{mathrsfs}
\usepackage{cleveref}
\usepackage{amsrefs}
\usepackage{microtype}

\newcommand{\bundle}{TM/F} 
\newcommand{\bundleConnection}{\nabla^{\bundle}} 
\newcommand{\tractor}[2]{\begin{pmatrix} {#1} \\ {#2} \end{pmatrix}}
\DeclareFontEncoding{LS1}{}{}
\DeclareFontSubstitution{LS1}{stix2}{m}{n}
\DeclareSymbolFont{symbols2}{LS1}{stix2frak}{m}{n}
\DeclareMathSymbol{\oplusrhrim}{\mathbin}{symbols2}{"FE}
\DeclareMathSymbol{\opluslhrim}{\mathbin}{symbols2}{"FD}
\DeclareSymbolFont{yhlargesymbols}{OMX}{yhex}{m}{n} 
\DeclareMathAccent{\yhwidehat}{\mathord}{yhlargesymbols}{"62}
\numberwithin{equation}{section}

\theoremstyle{plain}
\newtheorem{theorem}{Theorem}[section]
\newtheorem{lemma}{Lemma}[section]

\theoremstyle{remark}
\newtheorem*{remark}{Remark}
\theoremstyle{definition}
\newtheorem{definition}{Definition}[section]

\title[Elementary Relative Tractor Calculus\ldots]{Elementary Relative Tractor Calculus \\ for Legendrean Contact Structures}
\author[M.\ A.\ Wasilewicz]{Michał Andrzej Wasilewicz}
\address{Faculty of Mathematics, University of Vienna, Oskar-Morgenstern-Platz 1, 1090 Vienna, Austria}
\email{michal.wasilewicz@univie.ac.at}

\begin{document}
	\begin{abstract}
		For a manifold $M$ endowed with a Legendrean (or Lagrangean) contact structure $E\oplus F \subset TM$, we give an elementary construction of an invariant partial connection on the quotient bundle $TM/F$. This permits us to develop a na\"{i}ve version of relative tractor calculus and to construct a second order invariant differential operator, which turns out to be  the first relative BGG operator induced by the partial connection.
	\end{abstract}
	\maketitle
	
	\section{Introduction}
	In his pioneering paper \cite{lagrangeanContact}, Masaru Takeuchi introduced a notion of Lagrangean contact structures modelled on projectivised cotangent bundles. He showed that a choice of a projective class on a manifold gives rise to a refinement of the canonical contact structure. More precisely, he observed that given an $m$-dimensional manifold $M$, $m \geq 2$, the canonical contact structure $H$ on the projectivised cotangent bundle $\pi: \mathbb{P}(T^{*}M)\rightarrow M$ admits a decomposition into two Lagrangian subbundles $\ker(T\pi) \oplus E$, where the subbundle $E$ is  determined uniquely by a choice of a projective class. This follows from the fact that certain horizontal lifts agree for projectively equivalent connections. Hence, the bundle $E$ spanned on those lifts does not depend on a choice of a representative from the class (see \cite{lagrangeanContact}*{Lemma 4.2}).
	
	In the final part of his paper, M.\ Takeuchi carried out a rather arduous investigation into Cartan connections associated to Lagrangean contact structures. However, dealing with objects like second jet prolongations is highly non-trivial, and the modern approach of encoding the structure into parabolic geometries is preferred. The advantage is that the language of parabolic geometries provides a unified set of tools on the interface of differential geometry and representation theory. 
	
	We approach the study of Lagrangean contact structures in a more general setting than that of \cite{lagrangeanContact}. We consider an abstract contact structure admitting a splitting into two Legendrian subbundles $E\oplus F$ of equal rank. In particular, we do not impose any involutivity condition on either of the Legendrian subbundles. Of our main interest is a quotient of the tangent bundle by one of its Legendrian subbundles, say $F$. The crucial observation is that the resulting bundle $\bundle$ works well with the relative BGG machinery, which was developed in a sequence of papers \citelist{\cite{relativeBGG1} \cite{relativeBGG2}} by Andreas \v{C}ap and Vladim\'{\i}r Sou\v{c}ek. In the language of \cite{relativeBGG2}, the bundle $\bundle$ is an example of a relative tractor bundle. This is a completely novel approach and such a treatment has not appeared in the literature so far.
	
	This work consists of two main parts. The first part focuses on the study of distinguished partial connections on $\bundle$. This is mostly done by understanding the effect of a change of a contact form on the various data associated to it. Rather than referring to the results of the general theory, throughout the text we will gradually develop a simplified yet standalone version of so-called relative tractor calculus. The main result of this section is \Cref{thm:TractorConnection}, in which we derive a well-defined partial connection on $\bundle$. In the second part, we use the results of the previous section to construct first- and second-order invariant differential operators on completely reducible bundles. In particular, we give an elementary construction of the first relative BGG operator associated to the relative tractor bundle $\bundle$.
	
	Finally, it should be mentioned there is something of a problem of a nomenclature in this area. In the beginning, authors referred to a decomposition of a contact structure into two maximally $\mathcal{L}$-isotropic subbundles as Lagrangean contact structure.  On the other hand, in symplectic geometry a Lagrangian submanifold is the deep-rooted name for a maximal isotropic submanifold. Nowadays, we can observe an increasing tendency to use the term Legendrean contact structure instead of Lagrangean contact structure. 
	
	\section{Elementary relative tractor calculus}
	Let $M$ be a $(2n+1)$-dimensional $\mathscr{C}^{\infty}$-manifold endowed with a contact structure $(M, H)$. We write $TM = H \oplusrhrim Q$ for the quotient bundle $Q\coloneqq TM/H$, so that the Levi bracket induces a map $\mathcal{L}: \Lambda^{2}H \rightarrow Q$ with $\mathcal{L}(\eta, \xi) = \pi_{Q}([\eta, \xi])$ for $\eta, \xi \in \Gamma(H)$.
	
	\begin{definition}\label{def:legendreanContactStructure}
		A splitting of the contact subbundle into a direct sum of two rank-$n$ subbundles $H = E \oplus F$ such that $\mathcal{L}{|}_{E\times E} = 0$ and $\mathcal{L}{|}_{F\times F} = 0$, is called a \emph{Legendrean contact structure}.
	\end{definition}
	
	Notice that having chosen any of the two rank-$n$ subbundles of $H$, we can use it to form a quotient of the tangent bundle. More precisely, taking the bundle $F$ (resp.\ $E$) we obtain $\bundle\rightarrow M$ (resp.\ $TM/E \rightarrow M$). \Cref{def:legendreanContactStructure} does not impose any additional conditions on either Legendrean subbundle; however, in case both $E$ and $F$ are involutive, we arrive at a double fibration picture and picking one of the two subbundles is equivalent to choosing one side of the fibration. This arbitrariness should not be surprising in light of the relation to the general theory of relative BGG sequences, where the choice is an inherent part of the construction.
	
	Instrumental for the construction of relative tractor calculus is the observation that a choice of a contact form $\theta\in\Omega^{1}(M)$ splits $0\rightarrow H \rightarrow TM \rightarrow Q \rightarrow 0$, so the bundle $\bundle$ then decomposes into a direct sum $Q \oplus E$. Moreover, since $\operatorname{ker}(\theta) = H$, the contact form descends to the quotient bundle $Q$. In practice, it is also convenient to introduce a vector field $r$, called the \emph{Reeb vector field}, which is the unique vector field satisfying $\iota_{r}\,\theta = 1$ and $\iota_{r}\,d\theta = 0$. Using this, we are able to explicitly express the bundle maps reversing the arrows  as $\Gamma(Q)\ni\rho \mapsto \theta(\rho)r\in\Gamma(TM)$ and $\Gamma(TM)\ni t \mapsto (t - \theta(t)r) \in\Gamma(H)$. It is clear that being a section of the contact subbundle, $t - \theta(t)r$ can be decomposed corresponding to $H= E\oplus F$ as $(t-\theta(t)r)_{E} + (t-\theta(t)r)_{F}$. Throughout the text, we will follow this convention and denote the projections from $H$ onto $E$ and $F$ by a lower index.
	
	\begin{definition}\label{def:TractorSplitting}
		Given a contact form $\theta$ on $M$ and a section $t$ of $TM$, we define an isomorphism $\Gamma(TM/F) \ni (t + F) \mapsto {(t)}_{\theta} \in \Gamma(Q \oplus E)$ by the formula \[{(t)}_{\theta} \coloneqq \tractor{\pi_{Q}(t)}{\left(t-\theta(t)r\right)_{E}},\]
		where $r$ denotes the Reeb vector field corresponding to the contact form $\theta$.
	\end{definition}
	
	\begin{lemma}\label{lemma:CharacterisationOfReeb}
		Let $\theta$ and $\hat{\theta}$ be two contact forms on $M$ related by $\hat{\theta} = \mathrm{e}^{u}\theta$, where $u\in\mathscr{C}^{\infty}(M)$. Then, the $2$-form $d\theta$ and the Reeb vector field $r$ corresponding to $\theta$ transform in the following way: \[d\hat{\theta} = \mathrm{e}^{u}(du\wedge\theta + d\theta) \quad\textrm{and}\quad \hat{r} = \mathrm{e}^{-u}(r+\Upsilon),\]  where the section $\Upsilon\in\Gamma(H)$  is characterised by the identity $d\theta(\Upsilon, \xi) = du(\xi)$, for all $\xi\in\Gamma(H)$.
	\end{lemma}
	\begin{proof}
		An explicit computation immediately leads to $d\hat{\theta} = \mathrm{e}^{u}(du\wedge\theta + d\theta)$ and shows that $\mathrm{e}^{u}\hat{r} - r = \Upsilon \in\Gamma(H)$.
		
		Therefore, it only remains to find the expression for $\Upsilon$ in terms of $du$. For that purpose, let us take $\zeta\in\Gamma(TM)$ and consider $d\hat{\theta}(\hat{r}, \zeta)$ which, by the definition of Reeb vector field, vanishes identically
		\begin{align*}
			0 &= d\hat{\theta}(\hat{r}, \zeta) \\ 
			&= \left(du\wedge\theta + d\theta\right)(r + \Upsilon, \zeta) \\
			&= d\theta(\Upsilon, \zeta) + du(r+\Upsilon)\theta(\zeta) - du(\zeta).
		\end{align*}
		Substituting $r$ in place of $\zeta$, we readily see that $du(\Upsilon) = 0$; hence, the last equation simplifies to $-d\theta(\Upsilon, \zeta)= du(r)\theta(\zeta)-du(\zeta)$. Now, taking $\zeta$ to be a section of the bundle $H$ we get that $d\theta(\Upsilon, \zeta) = du(\zeta)$.
	\end{proof}
	
	\begin{lemma}\label{lemma:ChangeOfSplitting}
		Let $\theta$ and $\hat{\theta}$ be two contact forms on $M$ related by $\hat{\theta} = \mathrm{e}^{u}\theta$, where $u\in\mathscr{C}^{\infty}(M)$, and let $\Upsilon\in\Gamma(H)$ be the section from \Cref{lemma:CharacterisationOfReeb}. Then, the images of $t\in\Gamma(\bundle)$ under the isomorphisms $(\,\cdot\,)_{\theta}$ and $(\,\cdot\,)_{\hat{\theta}}$ are related by $$\yhwidehat{(\rho,\, \mu)}{}^{\intercal} = (\rho,\, \mu - \theta(\rho)\Upsilon_{E}){}^{\intercal}.$$
	\end{lemma}
	\begin{proof}
		Let $t$ be a section of the bundle $\bundle$ and set ${(t)}_{\hat{\theta}} = {(\hat{\rho}, \hat{\mu})}{}^{\intercal}$, for $\rho\in\Gamma(Q)$ and $\mu\in\Gamma(E)$ as in \Cref{def:TractorSplitting}. We need to check how this expression changes when passing to the related contact form $\theta = \mathrm{e}^{-u}\hat{\theta}$.
		
		Since the projection $\pi_{Q}: TM \rightarrow Q$ is a natural operation, we immediately see that $\rho$ remains unchanged. As for $\hat{\mu}$, we compute 
		\begin{align*}
			\hat{\mu} &= (t-\hat{\theta}(t)\hat{r})_{E}\\
			&= \left(t - \theta(t)(r+\Upsilon)\right)_{E} \\
			&= \left(t - \theta(t)r\right)_{E} - \theta(t){\Upsilon}_{E} \\
			&= \mu - \theta(\rho){\Upsilon}_{E}. \qedhere
		\end{align*}
	\end{proof}	
	
	\begin{lemma}
		Let $\theta$ and $\hat{\theta}$ be two contact forms on $M$ related by $\hat{\theta} = \mathrm{e}^{u}\theta$, where $u\in\mathscr{C}^{\infty}(M)$, with $r$ and $\hat{r}$ denoting the corresponding Reeb vector fields. Then, for any section $\xi$ of $F$ we get \[ [\xi, \hat{r}] = \mathrm{e}^{-u}\left([\xi, r+\Upsilon] - du(\xi)(r+\Upsilon)\right).\] In particular, \[ [\xi, \hat{r}]_{E} = \mathrm{e}^{-u}[\xi, r]_{E} + \mathrm{e}^{-u}d\theta(\xi, \Upsilon){\Upsilon}_{E} + \mathrm{e}^{-u}\big([\xi, \Upsilon] - \theta([\xi, \Upsilon])r\big)_{E}.\]
	\end{lemma}
	\begin{proof}
		A quick computation shows that
		\begin{align*}
			[\xi, \hat{r}] &= [\xi, \mathrm{e}^{-u}(r+\Upsilon)] \\
			&= \mathrm{e}^{-u}[\xi, r+\Upsilon] - \mathrm{e}^{-u}du(\xi)(r+\Upsilon) \\
			&= \mathrm{e}^{-u}\left([\xi, r+\Upsilon] + d\theta(\xi, \Upsilon)(r+\Upsilon)\right),
		\end{align*}
		which proves the first part of the statement. To see that the $E$-component of the Lie bracket has the desired form, recall that $[\zeta, r]$ is a section of $H$ for any $\zeta\in\Gamma(H)$. This, together with the fact that $\theta([\xi, \Upsilon])r = -d\theta(\xi, \Upsilon)r$ is precisely the $Q$-component of $[\xi, \Upsilon]$, proves the remaining part of the statement.
	\end{proof}
	
	The last crucial ingredient needed to complete the construction of relative tractor calculus are distinguished connections for Legendrean contact structures associated to a choice of a contact form. For our purposes, it is sufficient that we characterise the connections on the bundles $E$ and $Q$ in $F$-directions only. Of course, we require that those connections respect the Legendrean contact structure on the manifold. 
	
	Let $\theta$ be a contact form on $M$. Since the Reeb vector field $r$ associated to $\theta$ trivialises the bundle $Q$, we can define a partial connection on the line bundle $Q$ in the standard way. Explicitly, the connection  $\nabla^{Q}: \Gamma(F) \otimes \Gamma(Q) \rightarrow \Gamma(Q)$ is given by the formula \[\nabla_{\xi}^{Q} \rho = \big(\xi\cdot\theta(\rho)\big)\pi_{Q}(r). \]  
	
	As for the subbundle $E$, the general formulae for the distinguished connections can be found in literature (see \cite{parabolicGeometries}*{Proposition 5.2.14}). However, in the sequel we will give an ad-hoc description of the partial connection $\nabla^{E}: \Gamma(F) \otimes \Gamma(E) \rightarrow \Gamma(E)$ without referring to the general theory.
	
	\begin{lemma}\label{lemma:TranformationConnection}
		Let $\theta$ be a contact form on $M$ and $r$ the corresponding Reeb vector field. Moreover, let $\xi_{1}, \xi_{2}$ be two sections of the bundle $F$, $\eta$ a section of $E$, and $\rho$ any section of the bundle $Q$.
		\begin{enumerate}
			\item The following formula defines a partial connection $\nabla^{E}: \Gamma(F) \otimes \Gamma(E) \rightarrow \Gamma(E)$ \[d\theta(\nabla^{E}_{\xi_{1}}\eta, \xi_{2}) = d\theta([\xi_{1}, \eta], \xi_{2}).\]
			\item  Upon a change of a contact form $\hat{\theta}=\mathrm{e}^{u}\theta$, where $u\in\mathscr{C}^{\infty}(M)$, the partial connections on bundles $E$ and $Q$ transform in the following manner 
			\begin{align*}
				\widehat{\nabla}^{E}_{\xi_{1}}\eta &= {\nabla}^{E}_{\xi_{1}}\eta + d\theta(\xi_{1}, \eta){\Upsilon}_{E}\\ 
				\widehat{\nabla}^{Q}_{\xi_{1}}\rho &= {\nabla}^{Q}_{\xi_{1}}\rho + du(\xi_{1})\rho.
			\end{align*}
		\end{enumerate}
	\end{lemma}
	\begin{proof}\leavevmode
		\begin{enumerate}
			\item First of all, let us notice that nondegeneracy of $d\theta$ guarantees that the equation always completely determines the values of $\nabla^{E}$ in $F$-directions.
			
			Next, let us check that the right-hand side of the equation defines a partial connection. For $f\in\mathscr{C}^{\infty}(M, \mathbb{R})$ we get $d\theta([\xi_{1}, f\eta], \xi_{2}) = fd\theta([\xi_{1}, \eta], \xi_{2}) + (\xi_{1}\cdot f)d\theta(\eta, \xi_{2})$; hence, $\nabla^{E}_{\xi_{1}}f\eta = f\nabla_{\xi_{1}}^{E}\eta + (\xi_{1}\cdot f)\eta$. Furthermore, we see that vanishing of $d\theta$ on $\Lambda^{2}F$ implies that the expression is linear over $\mathscr{C}^{\infty}(M, \mathbb{R})$ in $\xi_{1}$, so the linearity conditions are indeed satisfied.
			
			\item We start with the connection $\nabla^{E}$ on $E$.
			\begin{align*}
				\mathcal{L}(\widehat{\nabla}^{E}_{\xi_{1}}\eta, \xi_{2}) &= -d\hat{\theta}([\xi_{1}, \eta], \xi_{2})\pi_{Q}(\hat{r})\\
				&= -\mathrm{e}^{u}\left(du\wedge\theta + d\theta\right)\!([\xi_{1}, \eta], \xi_{2})\,\pi_{Q}(\mathrm{e}^{-u}(r+\Upsilon)) \\
				&= -d\theta([\xi_{1}, \eta], \xi_{2})\pi_{Q}(r)+du(\xi_{2})\theta([\xi_{1}, \eta])\pi_{Q}(r) \\
				&= \mathcal{L}([\xi_{1}, \eta], \xi_{2}) + d\theta(\Upsilon, \xi_{2})d\theta(\xi_{1}, \eta)\pi_{Q}(r) \\
				&= \mathcal{L}(\nabla^{E}_{\xi_{1}}\eta, \xi_{2}) + \mathcal{L}(d\theta(\xi_{1}, \eta)\Upsilon, \xi_{2}),
			\end{align*}
			
			As for the connection $\nabla^{Q}$ on $Q$, a straightforward computation gives the desired result 
			\begin{align*}
				\widehat{\nabla}^{Q}_{\xi_{1}}\rho &= \big(\xi_{1}\cdot\hat{\theta}(\rho)\big)\pi_{Q}(\hat{r}) \\
				&= \mathrm{e}^{u}\big(du(\xi_{1})\theta(\rho) + \xi_{1}\cdot\theta(\rho)\big)\pi_{Q}(\mathrm{e}^{-u}(r+\Upsilon)) \\
				&= \nabla^{Q}_{\xi_{1}}\rho + d\theta(\Upsilon, \xi_{1})\rho.\qedhere
			\end{align*}
		\end{enumerate}
	\end{proof}

	\begin{remark}
		Notice that the formula for $\nabla^{E}$ from \Cref{lemma:TranformationConnection} can be rewritten explicitly as $\nabla^{E}_{\xi}{\eta} = ([\xi, \eta] - \theta([\xi, \eta])r)_{E}$. Indeed, since a choice of a contact form $\theta$ trivialises the quotient bundle $Q$, we can subtract from $[\xi, \eta]$ its projection onto $Q$, $\pi_{Q}([\xi, \eta]) =  \theta([\xi, \eta])r$. The resulting section of the contact subbundle can be decomposed, and by leaving out the $F$-component we arrive at the required formula.
	\end{remark}
	
	\begin{theorem}\label{thm:TractorConnection}
		Let let $\theta\in\Omega^{1}(M)$ be a contact form on $M$ and $r$ the corresponding Reeb vector field. Moreover, let $t \in \Gamma(\bundle)$, so that ${(t)}_{\theta}  = {(\rho, \mu)}^{\intercal}$ for a certain $\rho\in\Gamma(Q)$ and a certain $\mu\in\Gamma(E)$. Then, for any $\xi\in\Gamma(F)$ the formula 
		\begin{equation}\label{eq:TractorConnection}
			{\left(\bundleConnection_{\xi}t\right)}_{\theta}  \coloneqq  \tractor{\nabla^{Q}_{\xi}\rho + \mathcal{L}(\xi, \mu)}{\nabla^{E}_{\xi}\mu + \theta(\rho)[\xi, r]_{E}}
		\end{equation}
		gives rise to a well-defined partial connection $\bundleConnection$ on  $\bundle$ in $F$-directions.
	\end{theorem}
	\begin{proof}
		Let $t$ be as in the statement and fix $\hat{\theta}$ a contact form related to $\theta$ by $\hat{\theta} = \mathrm{e}^{u}\theta$, where $u\in\mathscr{C}^{\infty}(M)$. To prove the theorem we only need to show that the slot-wise expression for $\bundleConnection$ transforms well upon a change of a contact form.
		
		We start by computing the left-hand side of \cref{eq:TractorConnection} in the splitting determined by $\hat{\theta}$; for that purpose, we use the result of \Cref{lemma:ChangeOfSplitting}
		\begin{align*}
			\yhwidehat{\tractor{\nabla^{Q}_{\xi}\rho + \mathcal{L}(\xi, \mu)}{\nabla^{E}_{\xi}\mu +\theta(\rho)[\xi, r]_{E}} }
			&= {\tractor{{\nabla}^{Q}_{\xi}\rho + \mathcal{L}(\xi, \mu)}{\nabla^{E}_{\xi}\mu +\theta(\rho)[\xi, r]_{E} - \theta \left( {\nabla}^{Q}_{\xi}\rho + \mathcal{L}(\xi, \mu) \right)\Upsilon_{E}}}\\
			&= {\tractor{{\nabla}^{Q}_{\xi}\rho + \mathcal{L}(\xi, \mu)}{\nabla^{E}_{\xi}\mu +\theta(\rho)[\xi, r]_{E} - \left(\xi\cdot\theta(\rho) + \theta([\xi, \mu])\right){\Upsilon}_{E}}}.
		\end{align*}
		
		As for the right-hand side of \cref{eq:TractorConnection}, we take ${(t)}_{\hat{\theta}} = (\rho, \mu - \theta(\rho)\Upsilon_{E})^{\intercal}$ and use the results of \Cref{lemma:TranformationConnection} to express ${(\bundleConnection_{\xi}t)}_{\hat{\theta}}$ in terms of un-hatted operators. In order to make the computations easier, we start by writing out the general formula
		\[{\left(\bundleConnection_{\xi} t \right)}_{\widehat{\theta}} = 
		\tractor{\widehat{\nabla}^{Q}_{\xi}\rho + \mathcal{L}(\xi, \mu-\theta(\rho){\Upsilon}_{E})}{\widehat{\nabla}^{E}_{\xi}(\mu-\theta(\rho){\Upsilon}_{E})  +\hat{\theta}(\rho)[\xi,\hat{r}]_{E}}\]
		and compute each slot separately. The upper slot reads
		\begin{align*}
			\widehat{\nabla}^{Q}_{\xi}\rho + \mathcal{L}(\xi, \mu-\theta(\rho){\Upsilon}_{E}) &= \nabla^{Q}_{\xi}\rho + \mathcal{L}(\xi, \mu) + du(\xi)\rho - \mathcal{L}(\xi, \theta(\rho){\Upsilon}_{E}) \\
			&= \nabla^{Q}_{\xi}\rho + \mathcal{L}(\xi, \mu),
		\end{align*}
		while the bottom slot reads
		\begin{align*}
			\widehat{\nabla}^{E}_{\xi}(\mu-\theta(\rho){\Upsilon}_{E})  +\hat{\theta}(\rho)[\xi,\hat{r}]_{E} 
			=& \widehat{\nabla}^{E}_{\xi}\mu - \xi\cdot\theta(\rho)\Upsilon_{E} - \theta(\rho)\widehat{\nabla}^{E}_{\xi}{\Upsilon}_{E} + \theta(\rho)[\xi, r]_{E} \\ 
			&- \theta(\rho)\theta([\xi,\Upsilon]){\Upsilon}_{E} + \theta(\rho)\left([\xi, \Upsilon] - \theta([\xi, \Upsilon])r\right)_{E} \\
			=&\nabla^{E}_{\xi}\mu +\theta(\rho)[\xi, r]_{E} - \left(\xi\cdot\theta(\rho) + \theta([\xi, \mu])\right){\Upsilon}_{E} \\
			&+ \theta(\rho)([\xi, \Upsilon] - \theta([\xi, \Upsilon])r)_{E} - \theta(\rho)\nabla^{E}_{\xi}{\Upsilon}_{E},
		\end{align*}
		where the last line vanishes as $\nabla^{E}_{\xi}{\Upsilon}_{E} = ([\xi, \Upsilon] - \theta([\xi, \Upsilon])r)_{E}$.
	\end{proof}
	
	\begin{remark}
		In case the bundle $F$ is involutive, we know that $\bundle$ is endowed with a so-called \emph{Bott connection} $\nabla^{\mathcal{B}}$. This connection is given by a Lie bracket of a section of $F$ with a lift of a section of $\bundle$. It can be shown that the connections $\nabla^{\mathcal{B}}$ and $\bundleConnection$ coincide.
	\end{remark}
	
	\section{The first relative BGG operator}
	The fact that the bundle $\bundle$ has a nontrivial composition series is the main motivation for the sequel. The aim of this section is to construct a second-order invariant differential operator with values in simpler bundles. In the general theory, simple bundles are usually understood to be bundles induced by completely reducible representations. In our setting, these are exactly bundles which arise from $E$, $F$ and $Q$ by tensorial operations.
	
	Recall that non-degeneracy of the Levi bracket $\mathcal{L}: \Lambda^{2}H \rightarrow Q$ entails the existence of a bundle isomorphism $E \rightarrow F^{*}\otimes Q$, given by $\eta \mapsto \mathcal{L}(\eta, \,\cdot\,)$. What is more, since the natural projection $\operatorname{id}_{F^{*}}\otimes\, \pi_{Q}$ maps $F^{*}\otimes \bundle$ to $F^{*}\otimes Q$, we can compose it with the inverse of the isomorphism induced by $\mathcal{L}$ to obtain a natural bundle map from $F^{*}\otimes \bundle$ to $E$. Finally, applying the natural inclusion $E \rightarrow \bundle$ gives rise to a map from $F^{*}\otimes\bundle$ to $\bundle$, which we will henceforth denote by $\partial^{*}$. This notation is very natural as $\partial^{*}$ turns out to be a scalar multiple of the \emph{relative Kostant codifferential} customarily denoted by the same symbol (see \cite{relativeBGG1}*{Section 2.2}). 
	
	In the following, it is important to notice that the isomorphism $F^*\otimes Q\to E$ induced by the Levi bracket $\mathcal{L}$ gives us yet another useful identification. Namely, for any section $\rho$ of $Q$ we can view $\nabla^{Q}\rho$ as a section of $E$.
	
	\begin{lemma}\label{lemma:SplittingOperator}
		For every section $\rho$ of $Q$ there exists a unique section $t$ of $\bundle$ such that $\pi_{Q}(t) = \rho$ and $\partial^{*}\,\bundleConnection t = 0$. This gives rise to a first-order differential operator $\Gamma(Q)\ni \rho \mapsto S(\rho) \in \Gamma(\bundle)$; in a splitting corresponding to $\theta\in\Omega^{1}(M)$ it can be expressed as $S(\rho) = (\rho,\, \nabla^{Q}\rho)^{\intercal}$. 
	\end{lemma}
	\begin{proof}
		First of all, notice that since $\partial^{*}$ is an isomorphism, the condition that $\partial^{*}\,\bundleConnection t$ vanishes for a certain section $t \in\Gamma(\bundle)$ reduces to vanishing of the upper slot of $(\bundleConnection t)_{\theta}$ with respect to all contact forms $\theta$. On the other hand, \Cref{thm:TractorConnection} guarantees that if the upper slot of $(\bundleConnection t)_{\theta}$ vanishes identically for one contact form $\theta$, then it must vanish for all. 
		
		Now, let $\rho$ be as in the statement and let $t$ be any section of $\bundle$ such that $(t)_{\theta} = (\rho, \mu)^{\intercal}$, where $\mu\in\Gamma(E)$. The upper slot of $(\bundleConnection t)_{\theta}$ equals to $\nabla^{Q}\rho + \mathcal{L}(\,\cdot\, , \mu)$. Using the fact that $\mu$ can be seen both as a section of the bundle $E$ and a section of the bundle $F^{*}\otimes Q$, we conclude that $\nabla^{Q}\rho + \mathcal{L}(\,\cdot\, , \mu)$ vanishes if and only if $\mu = \nabla^{Q}\rho$. Thus, the operator $S$ is well-defined and it is of the form $S(\rho) = (\rho,\, \nabla^{Q}\rho)^{\intercal}$.
	\end{proof}
	
	Notice that the defining condition of $S$ implies that the composition $\bundleConnection\circ S$, seen as a map on $Q \otimes F$, has its values in the subbundle $E$ of $\bundle$. On the other hand, via the canonical isomorphism we can identify the bundle $E$ with $F^{*}\otimes Q$.
	
	\begin{definition}
		The second-order invariant differential operator $D \coloneqq \bundleConnection\circ S$ mapping sections of $Q$ to sections of  $F^{*}\otimes F^{*}\otimes Q$ is known as the \emph{first relative BGG operator} (see \cite{relativeBGG2}*{Definition 3.5}).
	\end{definition}
	
	In order to further investigate the properties of $D$, it is convenient to introduce a so-called \emph{Rho-tensor} $\mathrm{P}$. The name emphasises the relation with the usual $\mathrm{P}$-tensor present in the general theory (see \cite{parabolicGeometries}*{Section 5.1.2}).
	
	\begin{lemma}\label{lemma:RhoTensor}
		Let $\theta$ be a contact form on $M$ and $r$ the corresponding Reeb vector field. Then, for $\xi$ a section of $F$, the map \[\xi \mapsto \mathrm{P}(\xi) \coloneqq [\xi, r]_{E}\in\Gamma(E)\] is bilinear over $\mathscr{C}^{\infty}(M)$. Thus, for fixed $\theta$, we obtain a section $\mathrm{P}\in\Gamma(F^{*} \otimes E)$.		
	\end{lemma}
	\begin{proof}
		Let $\xi$ be as in the statement and take any $f\in\mathscr{C}^{\infty}(M)$. A direct computation shows that
		\begin{align*}
			\mathrm{P}(f\xi) &= [f\xi, r]_{E} \\
			&= (f[\xi, r] - df(r)\xi)_{E} \\
			&= f\,\mathrm{P}(\xi),
		\end{align*}
		where the last equation follows from the fact that $df(r)\xi$ is a section of $F$ and, therefore, vanishes when projected to $E$.
	\end{proof}
	
	Having all necessary ingredients at hand, we can proceed with a computation of the explicit expression for $D$. Recall that  $\nabla^{Q}\rho$ can be seen both as a section of the bundle $E$ and a section of the bundle $F^{*}\otimes Q$. In particular, an expression of the form $\nabla^{E}\nabla^{Q}\rho$ makes sense. Moreover, we can view $\mathrm{P}$ as a section of $F^*\otimes E\cong F^*\otimes F^*\otimes Q$.
	
	\begin{theorem}\label{thm:BggOperator}
		The first relative BGG-operator $D$ has values in the subbundle $\operatorname{Sym}^{2}(F^{*}) \otimes Q$ of $F^{*}\otimes F^{*} \otimes Q$. Moreover, for any contact form $\theta\in\Omega^{1}(M)$ and any $\rho\in\Gamma(Q)$, the operator takes the following form \[D(\rho) = \nabla^{E}\nabla^{Q}\rho + \theta(\rho)\mathrm{P},\] where $\mathrm{P}$ is as in \Cref{lemma:RhoTensor}.
	\end{theorem}
	\begin{proof}
		First, we want to find the image of $\bundleConnection S(\rho)$ for any $\rho\in\Gamma(Q)$. By the defining property of $S$, we know that the upper slot of  must vanish. As for the bottom slot, we only need to apply \Cref{thm:TractorConnection} to $S(\rho) = (\rho, \mu)^{\intercal}$. Simple calculations show that
		\begin{align*}
			\textrm{bottom slot of } \bundleConnection_{\xi_{1}} S(\rho) &= \nabla_{\xi_{1}}^{E}\mu + \theta(\rho)[\xi_{1}, r]_{E},
		\end{align*}
		where $\xi_{1}$ is any section of the bundle $F$ and the last summand is $\theta(\rho)\mathrm{P}(\xi_{1})$. However, it is not immediately clear that this is indeed a section of $F^{*}\otimes F^{*}\otimes Q$. To see that, we need to rewrite the equation  using the canonical isomorphism induced by $\mathcal{L}$ as
		\begin{align*}
			D(\rho)(\xi_{1}, \xi_{2}) &= \mathcal{L}(\nabla^{E}_{\xi_{1}}\mu + \theta(\rho)[\xi_{1}, r], \xi_{2})\\
			&= -d\theta([\xi_{1}, \mu] + \theta(\rho)[\xi_{1}, r], \xi_{2})\pi_{Q}(r).
		\end{align*}
		
		Thus, it only remains to show that the tensor $D(\rho)$ is symmetric with respect to its $F$-entries. For the sake of clarity, we will consider the two components $d\theta([\xi_{1}, \eta], \xi_{2})$ and $\theta(\rho)d\theta([\xi_{1}, r], \xi_{2})$ separately. In either case, however, the argument follows from the identity $d^{2}\theta = 0$. Let $\zeta$ be any section of $TM$ and $\xi_{1}, \xi_{2}$ be two sections of $F$. The usual invariant formula for exterior derivative gives
		\begin{multline}\label{eq:TransformationSymmetric}
			0 = d^{2}\theta(\xi_{1}, \zeta, \xi_{2}) = d\theta([\xi_{1}, \zeta], \xi_{2}) - d\theta([\xi_{2}, \zeta], \xi_{1}) \\ + \xi_{2}\cdot d\theta(\zeta,\xi_{1}) - \xi_{1}\cdot d\theta(\zeta,\xi_{2}) -d\theta([\xi_{1}, \xi_{2}], \zeta).
		\end{multline}
	
		To deal with the component $d\theta([\xi_{1}, r], \xi_{2})$ it is enough that we substitute $r$ for $\zeta$ in \cref{eq:TransformationSymmetric}. We can immediately see that vanishing of the bottom line follows as a consequence of $\iota_{r}d\theta=0$. This leads to $0 = d\theta([\xi_{1}, r], \xi_{2}) - d\theta([\xi_{2}, r], \xi_{1})$ and shows that the component is symmetric.
		
		Analogously, we show that the component $d\theta([\xi_{1}, \mu], \xi_{2})$ is symmetric by substituting $\mu$ for $\zeta$ in \cref{eq:TransformationSymmetric}. The definition of $\mu$ implies that $d\theta(\mu, \xi) = -\xi\cdot \theta(\rho)$, and the equation simplifies to
		\begin{align*}
			0 =& d\theta([\xi_{1}, \mu], \xi_{2}) - d\theta([\xi_{2}, \mu], \xi_{1}) \\ 
				&- \xi_{2}\cdot\xi_{1}\cdot \theta(\rho) + \xi_{1}\cdot \xi_{2} \cdot \theta(\rho) - [\xi_{1}, \xi_{2}] \cdot \theta(\rho).
		\end{align*}
		By expanding the Lie bracket $[\xi_{1}, \xi_{2}]$ we obtain two terms, which cancel out the bottom line of the equation. This leads to the equality $d\theta([\xi_{1}, \mu], \xi_{2}) = d\theta([\xi_{2}, \mu], \xi_{1})$.
	\end{proof}

		
\end{document}